\documentclass[12pt,a4paper]{article}
\usepackage[T2A]{fontenc}
\usepackage[cp1251]{inputenc}
\usepackage[english]{babel}
\usepackage{amsmath,amsfonts,amssymb, amsthm}
\usepackage{cite}
\usepackage{xcolor}
\usepackage{hyperref}
\usepackage[left=2.5cm,right=1.5cm,
 top=2cm,bottom=2cm,bindingoffset=0cm]{geometry}

\theoremstyle{plain}
\newtheorem{theorem}{Theorem}
\newtheorem{proposition}{Proposition}
\newtheorem{pcorollary}{Corollary}[proposition]
\newtheorem{lemma}{Lemma}
\newtheorem{corollary}{Corollary}[theorem]

\theoremstyle{definition}
\newtheorem{example}{Example}
\newtheorem{definition}{Definition}
\begin{document}

{\large
\textbf{\textsc{Finite groups with $\mathbb{P}$-subnormal and strongly permutable subgroups}}}

\medskip

\textbf{\textsc{
V.\,S. Monakhov\footnote[1]{Scorina Gomel State University, Gomel,
Belarus; \texttt{Victor.Monakhov@gmail.com}},
I.\,L. Sokhor\footnote[2]{Brest State A.S. Pushkin University, Brest,
Belarus; \texttt{irina.sokhor@gmail.com}}
}}

\bigskip

{\footnotesize

\textbf{Abstract:} Let~$H$ be a subgroup of a group~$G$.
The permutizer $P_G(H)$ is the subgroup generated
by all cyclic subgroups of~$G$ which permute with $H$.
A subgroup~$H$ of a group~$G$ is strongly permutable in~$G$ if $P_U(H)=U$
for every subgroup~$U$ of~$G$ such that~$H\le U\le G$.
We investigate groups with $\mathbb{P}$-subnormal or
strongly permutable Sylow and primary cyclic subgroups.
In particular, we prove that groups
with all strongly permutable primary cyclic subgroups are supersoluble.


}

\section{Introduction}

All groups in this paper are finite.
A group of prime power order is called a primary group.

Let~$H$ be a subgroup of a group~$G$.
The permutizer of~$H$ in $G$ is the subgroup generated
by all cyclic subgroups of~$G$ which permute with $H$, i.\,e.
$$
P_G(H)=\left\langle x\in G\mid \langle x\rangle H=H\langle x\rangle
 \right\rangle.
$$
The permutizer $P_G(H)$ contains the normalizer $N_G(H)$, see~\cite[p.~26]{Wein1982}.
X.~Liu and Y.~Wang~\cite{LiuWang05} proved that
a group $G$ has a Sylow tower of supersoluble type
if $P_G(X)=G$ for every Sylow subgroup~$X$ of~$G$.
A.\,F.~Vasil'ev, V.\,A.~Vasil'ev and T.\,I.~Vasil'eva~\cite{VVV2014}
described the structure of a group~$G$ in which $P_Y(X)=Y$
for every Sylow subgroup~$X$ of~$G$ and every subgroup~$Y\ge X$.
They proposed the following notation.

\begin{definition}\label{dper} 
A subgroup~$H$ of a group~$G$ is

$(1)$ {\sl permutable} in~$G$ if $P_G(H)=G$;

$(2)$ {\sl strongly permutable} in~$G$ if $P_U(H)=U$
for every subgroup~$U$ of~$G$ such that~$H\le U\le G$.
\end{definition}

We note that a quasinormal subgroup is strongly permutable.
In the symmetric group~$S_n$, $n\in \{3,4,6\}$,
a Sylow 2-subgroup is strongly permutable, but it is not quasinormal.

A.\,F~Vasil'ev, V.\,A.~Vasil'ev and V.\,N.~Tyutyanov~\cite{VVT2010}
proposed the following notation.

\begin{definition}
Let $\mathbb{P}$ be the set of all primes.
A subgroup $H$ of a group $G$ is called
$\mathbb{P}$\nobreakdash-\hspace{0pt}{\sl subnormal} in~$G$
if there is a subgroup chain
$$
H=H_0\le H_1\le \ldots \le H_n=G \eqno (1)
$$
such that $|H_i:H_{i-1}|\in \mathbb{P}\cup \{1\}$ for every $i$.
\end{definition}

The class of all groups with $\mathbb{P}$-subnormal Sylow subgroups
is denoted by~$\rm{w}\mathfrak U$ and the class of all groups
with $\mathbb{P}$-subnormal primary cyclic subgroups is denoted by~$\rm{v}\mathfrak U$.
These classes are quite well studied~\cite{VVT2010,mkRic,Mon2016SMJ}.
In particular, these classes are subgroup-closed saturated formations.

In a soluble group, a $\mathbb{P}$\nobreakdash-\hspace{0pt}subnormal
Hall subgroup (in particular, a Sylow subgroup)
is strongly permutable~\cite[3.8]{VVV2014}.
We prove that in a soluble group, the converse is true, see Proposition~\ref{l4}.
As a result we obtain new criteria for the supersolubility of a group
and also~\cite[Theorem A]{czl}: {\sl a group~$G\in\mathrm{w}\mathfrak{U}$
if and only if every Sylow subgroup is $\mathbb{P}$-subnormal
or strongly permutable in~$G$.}

In the following theorem, we enumerate all simple non-abelian groups
with a $\mathbb{P}$-subnormal or strongly permutable Sylow subgroup.

\begin{theorem}
 Let $G$ be a simple non-abelian group
 and let $R$ be a Sylow $r$-subgroup of $G$.

 $(1)$ If $R$ is $\mathbb{P}$-subnormal in~$G$,
 then $r=2$ and $G$ is isomorphic to $L_2(7)$, $L_2(11)$
 or $L_2(2^m)$ and~$2^m+1$ is a prime.

 $(2)$ If $R$ is strongly permutable and $\mathbb{P}$-subnormal in~$G$,
 then $r=2$ and $G\cong L_2(7)$.
\end{theorem}

For a primary cyclic subgroups we prove two following theorems.

\begin{theorem}
 If all primary cyclic subgroups of a group $G$ are
 strongly permutable, then~$G$ is supersoluble.
\end{theorem}

\begin{theorem}
 If every primary cyclic subgroup of a group $G$ is
 $\mathbb{P}$\nobreakdash-\hspace{0pt}subnormal or strongly permutable,
 then~$G\in \rm{v}\mathfrak U$.
\end{theorem}

\section{Preliminaries}

Let $G$ be a group. We use $\pi (G)$ to denote the set of all prime
devisors of $|G|$.
If $r$ is a maximal element of $\pi(G)$, then we write
$r=\mathrm{max}\ \pi(G)$.
By $H\leq G$ ($H< G$, $H\lessdot \,G$, $H\lhd G$) we denote a
(proper, maximal, normal)
subgroup $H$ of $G$.

We use the GAP system~\cite{gap} to build examples.
Note that GAP package Permut~\cite{gap_perm} is especially useful
for testing subgroups permutability.

\begin{lemma}\label{KPProp}
Let $H$ and $L$ be subgroups of a group $G$
and let $N$ be a normal subgroup of~$G$.

$(1)$ If $H$ is $\mathbb{P}$-subnormal in~$G$,
 then $H\cap N$ is $\mathbb{P}$-subnormal in~$N$
 and $HN/N$ is $\mathbb{P}$-subnormal in~$G/N$~\textup{\cite[Lemma~3]{mkRic}}.

$(2)$ If $H$ is $\mathbb{P}$-subnormal in a soluble group $G$
 and $U\le G$, then $H\cap U$ is $\mathbb{P}$-subnormal
 in~$U$~\textup{\cite[Lemma~4\,(1)]{mkRic}}.

$(3)$ If $H\leq L$, $H$ is $\mathbb{P}$-subnormal in $L$
 and $L$ is $\mathbb{P}$-subnormal in $G$,
 then $H$ is $\mathbb{P}$-subnormal in $G$~\textup{\cite[Lemma~3]{mkRic}}.
\end{lemma}

\begin{lemma}[{\cite[Lemma~2\,(2)]{mkRic}}]\label{UPsn}
Every subgroup of a supersoluble group is $\mathbb{P}$-subnormal.
\end{lemma}

\begin{lemma}\label{KPMax}
 Let $H$ be a $\mathbb{P}$-subnormal subgroup of a group $G$.
 If $r=\mathrm{max}\ \pi(G)$, then $O_r(H)\leq O_r(G)$.
\end{lemma}

\begin{proof}
By the hypothesis, there is a subgroup chain
\[
H=H_0< H_1< H_2<\ldots< H_n=G
\]
such that for every $i$, $|H_{i}:H_{i-1}|\in\mathbb{P}$.
Assume that  $O_r(H)\leq O_r(H_{i-1})$, we prove $O_r(H)\leq O_r(H_{i})$.
Let $H_{i-1}=A$ and $H_{i}=B$.
If $A$ is normal in $B$, then $O_r(A)$ is subnormal in $B$,
and $O_r(H)\leq O_r(A)\leq O_r(B)$.
If $A$ is not normal in $B$, then $|B:A|=q\in\mathbb{P}$ and $A=N_B(A)$.
Consider  the representation of $B$ by permutations on the
right cosets of~$A$~\cite[I.6.2]{hup}.
Note that $B/A_B$ is isomorphic to a subgroup
of the symmetric group $S_q$ and $|B/A_B:A/A_B|=q$.
Since $|S_q|=q!=(q-1)!q$ and $|B/A_B|$ divides $|S_q|$,
we get $|A/A_B|$ divides $(q-1)!$. As $q\in\pi(B)\subseteq\pi(G)$
and $r=\mathrm{max}\ \pi(G)$, we have $q\leq r$ and $A/A_B$ is an $r^\prime $-group.
Therefore $O_r(A)\leq A_B$. Since $O_r(A)$ is normal in $A_B$,
we get $O_r(A)$ is subnormal in $B$ and $O_r(H)\leq O_r(A)\leq O_r(B)$.
Hence $O_r(H)\leq O_r(G)$ by induction.
\end{proof}

The following lemma contains permutable and strongly permutable subgroups
properties we need.

\begin{lemma}\label{PerProp}
Let $H$ be a subgroup of a group $G$
and let $N$ be a normal subgroup of~$G$.

 $(1)$~If $H$ is (strongly) permutable in $G$,
 then $HN/N$ is (strongly) permutable in $G/N$
 \textup{\cite[lemma 3.2.\,(1),(4)]{VVV2014}}.

 $(2)$~If $N \leq H$, then $H$ is (strongly) permutable in $G$
 if and only if $H/N$ is (strongly) permutable in $G/N$.

 $(3)$~If $H$ is strongly permutable in $G$ and $H\le U$,
 then $H$ is strongly permutable in $U$.
\end{lemma}

\begin{proof}
$(2)$ This statement is known for permutable subgroups~\cite[lemma 3.2\,(3)]{VVV2014}.
We prove it for strongly permutable subgroups.
If $H$ is a strongly permutable subgroup of $G$,
then in view of Statement~$(1)$, $H/N$ is strongly permutable in $G/N$.
Conversely, let $H/N$ be strongly permutable in $G/N$
and let $A$ be a subgroup of $G$ containing $H$.
Then $P_{A/N}(H/N)=A/N$. In view of~\cite[Lemma~3.6]{VVV2014},
$P_{A/N}(H/N)=P_A(H)/N$ and $P_A(H)=A$,
hence $H$ is strongly permutable in $G$.

$(3)$~This is evident in view of Definition~\ref{dper}\,(2).
\end{proof}

\begin{lemma}\label{PG=NG}
 Let $r=\mathrm{max}\ \pi(G)$
 and let $R$ be a Sylow $r$-subgroup of a group $G$.
 Then $N_G(R)=P_G(R)$. In particular,
 if $R$ is permutable in $G$, then $R$ is normal in $G$.
\end{lemma}

\begin{proof}
Let $x\in G$ and $R\langle x\rangle=\langle x\rangle R$.
It is clear $\langle x\rangle=\langle x_1\rangle\times\langle x_2\rangle$,
where
$\langle x_1\rangle$ is a Sylow $r$-subgroup of $\langle x\rangle$ and
$\langle x_2\rangle$ is a Hall $r^\prime$-subgroup of $\langle x\rangle$.
In view of~\cite[VI.4.7]{hup}, $R=R\langle x_1\rangle$,
therefore $R\langle x\rangle=R\langle x_2\rangle $.
Now, all Sylow $r^\prime $-subgroups of $R\langle x\rangle$ is cyclic.
As $r=\mathrm{max}\ \pi(R\langle x\rangle)$,
it implies that $R$ is normal in~$R\langle x\rangle$ by \cite[IV.2.7]{hup},
and $\langle x\rangle\leq N_G(R)$.
Since $P_G(R)$ is generated by elements $x$ such that
$R\langle x\rangle=\langle x\rangle R$, we conclude $P_G(R)=N_G(R)$.
\end{proof}

\begin{lemma}\label{SylTower}
 If every Sylow subgroup of a group $G$ is
 $\mathbb{P}$-subnormal or permutable,
 then $G$ has a Sylow tower of supersoluble type.
\end{lemma}

\begin{proof}
Let $R$ be a Sylow $r$-subgroup of a group $G$ for $r=\mathrm{max}\ \pi(G)$.
If $R$ is $\mathbb{P}$-subnormal in $G$,
then $R$ is normal in $G$ by Lemma~\ref{KPMax}.
If $R$ is permutable in $G$, then $R$ is normal in $G$ by Lemma~\ref{PG=NG}.
Hence $R$ is normal in~$G$.

Let $\overline{Q}$ be a Sylow $q$-subgroup of $\overline{G}=G/R$.
Then $\overline{Q}=QR/R$ for a Sylow $q$-subgroup $Q$ of~$G$.
If $Q$ is $\mathbb{P}$-subnormal in $G$,
then $\overline{Q}$ is $\mathbb{P}$-subnormal in $\overline{G}$
in view of Lemma~\ref{KPProp}\,(1).
If $Q$ is permutable in $G$, then $\overline{Q}$ is permutable in $\overline{G}$
by Lemma~\ref{PerProp}\,(1).
Thus, the hypothesis holds for  $\overline{G}$,
and by induction, $\overline{G}$ has a Sylow tower of supersoluble type.
Therefore $G$ also has a Sylow tower of supersoluble type.
\end{proof}

\begin{lemma}[{\cite[Lemma~2.1]{QQW2013}}]\label{MC}
 Let $M$ be a maximal subgroup of a soluble group $G$,
 and assume that $G=MC$ for a cyclic subgroup $C$.
 Then $|G : M|$ is a prime or $4$.
 Also, if $|G : M| = 4$, then $G/M_G = S_4$.
\end{lemma}

We will also repeatedly use the following statement.

\begin{lemma}[{\cite[Lemma~2.2]{mkRic}}]\label{PrimitiveGr}
 Let $\mathfrak{F}$ be a saturated formation and let $G$ be a group.
 Suppose that $G\notin\mathfrak{F}$ but $G/N \in \mathfrak{F}$
 for any normal subgroup $N$ of $G$, $N\neq 1$. Then $G$ is a primitive group.
\end{lemma}

\section{Groups with permutable and $\mathbb{P}$-subnormal Sylow subgroups}

\setcounter{theorem}{0}

\begin{proposition}\label{l4}
 Let $G$ be a soluble group and let $H$ be a Hall subgroup.
 Then $H$ is $\mathbb{P}$-subnormal in $G$ if and only if
 $H$ is strongly permutable in $G$.
\end{proposition}

\begin{proof}
Let $H$ be $\mathbb{P}$-subnormal in $G$.
According to~\cite[3.8]{VVV2014}, $H$ is strongly permutable in $G$.
For completeness, we give the proof of this statement.
We use induction on the order of $G$.
Since $H$ is $\mathbb{P}$-subnormal in $G$,
there is a maximal subgroup $M$ of $G$ such that $H\leq M$,
$|G:M|\in\mathbb{P}$ and $H$ is $\mathbb{P}$-subnormal in $M$.
By induction, $H$ is strongly permutable in $M$ and
$M=P_M(H)\leq P_G(H)$. Since $M$ is a maximal subgroup of $G$,
we assume $P_G(H)=M$. Suppose that $M_G\neq 1$ and
$L$ is a minimal normal subgroup of $G$ that is contained in $M_G$.
According to Lemma~\ref{KPProp}\,(1), $HL/L$ is $\mathbb{P}$-subnormal in $G/L$,
and by induction, $HL/L$ is permutable in $G/L$.
Hence $HL$ is permutable in $G$ in view of Lemma~\ref{PerProp}\,(3).
Since $G$ is soluble, we conclude $L$ is an elementary abelian $q$-group
for some $q\in\pi(G)$.
If $q\in\pi(H)$, then $HL=H$ and $H$ is permutable in $G$, a contradiction.
Therefore we can assume that $q\notin \pi(H)$.
Since $HL$ is permutable in $G$, then $P_G(HL)=G$ and
there is $x\in G\setminus M$ such that
$\langle x\rangle HL=HL\langle x\rangle=A$.
Suppose that $A$ is a proper subgroup of $G$.
As $H$ is $\mathbb{P}$-subnormal in $G$,
by Lemma~\ref{KPProp}\,(2), it follows that
$H$ is $\mathbb{P}$-subnormal in $A$,
and by induction, $H$ is permutable in $A$.
Therefore $A=P_A(H)\leq P_G(H)=M$ and $x\in M$, a contradiction.
Hence $G=\langle x\rangle HL$.
If $L\leq\Phi(G)$, then $G=\langle x\rangle H$ and
$x\in P_G(H)=M$, a contradiction.
Consequently, $L$ is not contained in $\Phi (G)$ and
there is a maximal subgroup $K$ of $G$
that does not contain $L$.
In that case, $G=LK$ and we can assume $H\leq K$.
By induction, $H$ is permutable in $K$ and $K=P_K(H)\leq P_G(H)=M$.
Hence we get $M=K$ and $L\leq K$, a contradiction.
Thus, $M_G=1$ and $G$ is a primitive group.
Consequently, $G=N\rtimes M$, where $N=F(G)$ is
a unique minimal normal subgroup of $G$.
Since $|G:M|\in\mathbb{P}$, we deduce
$N$ is a cyclic subgroup and $N\leq P_G(H)=M$, a contradiction.
Thus $H$ is permutable in $G$, and in view of Lemma~\ref{KPProp}\,(2),
$H$ is strongly permutable in~$G$.

Conversely, let $H$ be a Hall strongly permutable subgroup
of a soluble group $G$. Using induction on the order of $G$
we prove that $H$ is $\mathbb{P}$-subnormal in $G$.
Let $H\leq M\lessdot G$. By Lemma~\ref{PerProp}\,(3),
$H$ is strongly permutable in $M$, and by induction,
$H$ is $\mathbb{P}$-subnormal in $M$.
If $|G:M|\in\mathbb{P}$, then $H$ is $\mathbb{P}$-subnormal in $G$
by Lemma~\ref{KPProp}\,(3).
Hence we can assume that $|G:M|\notin\mathbb{P}$,
in particular, $M$ is not normal in $G$.
According to Lemma~\ref{PerProp}\,(1), $G/M_G$ contains
a strongly permutable Hall subgroup $HM_G/M_G$.
As $HM_G\leq M$, we obtain that $H$ is $\mathbb{P}$-subnormal in $HM_G$
by induction. If $M_G\neq 1$, then $HM_G/M_G$ is $\mathbb{P}$-subnormal in $G/M_G$
by induction. By Lemma~\ref{KPProp}\,(1), $HM_G$ is $\mathbb{P}$-subnormal in $G$,
and $H$ is $\mathbb{P}$-subnormal in $G$ in view of Lemma~\ref{KPProp}\,(3).
Therefore we can assume that $M_G=1$.
Since $G$ is soluble, we get $G=N\rtimes M$,
$N=F(G)=C_G(N)=O_p(G)$ is a unique minimal normal subgroup in $G$.
Let $HN<G$. By induction, $HN/N$ is $\mathbb{P}$-subnormal in $G/N$,
and $H$ is $\mathbb{P}$-subnormal in $G$.
Finally, we consider the case, when $H=M$ is a Hall subgroup.
By the hypothesis, there is $x\in G\setminus H$
such that $\langle x\rangle H=H\langle x\rangle=G$.
Let $\langle x\rangle=\langle x_1\rangle \times\langle x_2\rangle$,
where $\langle x_1\rangle$ is a $p$-subgroup and
$\langle x_2\rangle$ is a $p^\prime $-subgroup.
Since $N$ is a normal Sylow $p$-subgroup of $G$,
we conclude $\langle x_1\rangle\leq N$ and $|x_1|=p$.
According to~\cite[VI.4.6]{hup}, $H=\langle x_2\rangle H$.
Now, $G=\langle x\rangle H=\langle x_1\rangle H$  and $|G:H|=p$.
\end{proof}

\begin{pcorollary}\cite[Theorem~A]{czl}\label{th4}
 If every Sylow subgroup of a group~$G$ is
 $\mathbb{P}$-subnormal or strongly permutable,
 then $G\in\mathrm{w}\mathfrak{U}$.
 Conversely, if $G\in\mathrm{w}\mathfrak{U}$,
 then every Sylow subgroup is $\mathbb{P}$-subnormal
 and strongly permutable in~$G$.
\end{pcorollary}

\begin{proof}
Assume that every Sylow subgroup of $G$ is $\mathbb{P}$-subnormal or strongly permutable.
By Lemma~\ref{SylTower}, $G$ has a Sylow tower of supersoluble type,
which means that $G$ is soluble. Hence by Proposition~\ref{l4},
every strongly permutable Sylow subgroup of $G$
is $\mathbb{P}$-subnormal and $G\in\mathrm{w}\mathfrak{U}$.

Conversely, let $G\in\mathrm{w}\mathfrak{U}$.
By the definition of $\mathrm{w}\mathfrak{U}$,
every Sylow subgroup is $\mathbb{P}$-subnormal in~$G$.
Since~$G$ is soluble, according to Proposition~\ref{l4},
every Sylow subgroup is strongly permutable in~$G$.
\end{proof}

\begin{pcorollary}\label{c42}
Let $G$ be a group. The following statements are equivalent.

$(1)$ $G$ is supersoluble.

$(2)$ Every Hall subgroup of $G$ is $\mathbb{P}$-subnormal or strongly permutable.

$(3)$ Every Hall subgroup of $G$ is $\mathbb{P}$-subnormal or permutable.
\end{pcorollary}

\begin{proof}
$(1)\Rightarrow (2)$: If $G$ is supersolvable,
then by Lemma~\ref{UPsn}, every subgroup in $G$ is $\mathbb{P}$-subnormal,
and every Hall subgroup is $\mathbb{P}$-subnormal.
Since~$G$ is soluble, we conclude that every Hall  subgroup
is strongly permutable in~$G$ in view of Proposition~\ref{l4}.

$(2)\Rightarrow (3)$: It is evident since every strongly permutable subgroup is permutable.

$(3)\Rightarrow (1)$:
Assume that every Hall subgroup of a group $G$ is $\mathbb{P}$-subnormal or permutable.
By Lemma~\ref{SylTower}, $G$ has a Sylow tower of supersoluble type,
and for $r=\mathrm{max}\ \pi(G)$, a Sylow $r$-subgroup $R$ is normal in $G$.

Let $N$ be a normal subgroup of $G$, $N\neq 1$,
and let $\overline{H}$ be a Hall $\pi$-subgroup of
$\overline{G}=G/N$ for $\pi\subseteq\pi(G)$.
Then $\overline{H}=HN/N$  for a Hall $\pi$-subgroup $H$ of $G$.
If $H$ is $\mathbb{P}$-subnormal in $G$,
then $\overline{H}$ is $\mathbb{P}$-subnormal in $\overline{G}$
by Lemma~\ref{KPProp}\,(1). If $H$ is permutable in $G$,
then $\overline{H}$ is permutable in $\overline{G}$ by Lemma~\ref{PerProp}\,(1).
Thus the hypothesis holds for $\overline{G}$,
and by induction, $\overline{G}\in\mathfrak{U}$.
As $\mathfrak{U}$ is a subgroup-closed satuarted formation,
we deduce $G$ is a primitive group by Lemma~\ref{PrimitiveGr}.
According to properties of primitive groups, $\Phi(G)=1$,
$G=R\rtimes M$, $R=F(G)$ is a minimal normal  subgroup in $G$, $|R|>r$,
$M$ is a maximal subgroup in $G$, $M\in\mathfrak{U}$.
Note that $M$ is  a Hall subgroup. If $M$ is $\mathbb{P}$-subnormal
in $G$, then $|G:M|=r=|R|$, a contradiction.
Suppose that $M$ is permutable in $G$, i.\,e. $P_G(M)=G$.
In that case, there is $x\in G\setminus M$ such that $G=M\langle x\rangle$.
In view of Lemma~\ref{MC}, $|G:M|=4$,
but $r=\mathrm{max}\ \pi(G)$. Hence $G$ is a $2$-group.
\end{proof}

\begin{pcorollary}\label{BiU}
 If every Sylow subgroup of a biprimary group $G$ is
 $\mathbb{P}$-subnormal or permutable,
 then $G$ is supersoluble.
 Conversely, in a supersoluble biprimary group
 every Sylow subgroup is $\mathbb{P}$-subnormal
 and strongly permutable.
\end{pcorollary}

According to Proposition~\ref{l4}, for a Hall subgroup of a soluble group,
$\mathbb{P}$-subnormality and strongly permutability are equivalent.
In simple groups, this is not true.

\begin{example}
 In $L_2(8)$, a Hall $\{2,7\}$-subgroup is strongly permutable,
 but it is not $\mathbb{P}$-subnormal.
\end{example}

\begin{example}
 In $L_2(9)$, a Sylow $2$-subgroup is strongly permutable,
 but it is not $\mathbb{P}$-subnormal.
\end{example}

\begin{example}
 In $L_2(5)$, a Sylow 2-subgroup is $\mathbb{P}$-subnormal,
 but it is not permutable.
\end{example}

\begin{theorem}\label{th0}
 Let $G$ be a simple non-abelian group
 and let $R$ be a Sylow $r$-subgroup of $G$.

 $(1)$ If $R$ is $\mathbb{P}$-subnormal in~$G$,
 then $r=2$ and $G$ is isomorphic to $L_2(7)$, $L_2(11)$
 or $L_2(2^m)$ and~$2^m+1$ is a prime.

 $(2)$ If $R$ is strongly permutable and $\mathbb{P}$-subnormal in~$G$,
 then $r=2$ and $G\cong L_2(7)$.
\end{theorem}

\begin{proof}
Since~$R$ is $\mathbb{P}$-subnormal in~$G$,
there is a subgroup chain
$$
R=H_0\le H_1\le H_2\le\ldots\le H_{n-1}=H\le H_n=G
$$
such that $|H_{i+1}:H_i|\in\mathbb{P}$.
It is clear that~$R$ is $\mathbb{P}$-subnormal in~$H$.
Let $|G:H|=p$. Since $H_G=1$,
the representation of~$G$ on the set of left  cosets
by~$H$ is exactly of degree $p$~\cite[I.6.2]{hup} and~$G$
is ismorphic to a subgroup of the symmetric group~$S_p$ of order~$p$
for $p=\mathrm{max}\ \pi(G)$, $H$ is a Hall $p^\prime$-subgroup of $G$.
From Lemma~\ref{KPMax} it follows that a Sylow $p$-subgroup of~$G$
is not $\mathbb{P}$-subnormal in~$G$, therefore~$r<p$.
Since the unit subgroup is $\mathbb{P}$-subnormal in~$R$,
the unit subgroup is $\mathbb{P}$-subnormal in~$H$ and in~$G$.
According to \cite{kaz}, \cite[p. 342]{cs}, $G$ is
isomorphic to one of the following groups
$$
L_2(7), \ L_2(11), \ L_3(3), \ L_3(5), \ L_2(2^m),
2^m+1 \mbox{ is prime}.
$$

Let $G\cong L_2(7)$. Then $|G|=2^3\cdot 3\cdot 7$, $p=7$, $H\cong S_4$.
In $S_4$, a Sylow $2$-subgroup is $\mathbb{P}$-subnormal,
a Sylow $3$-subgroup is not $\mathbb{P}$-subnormal, therefore~$r=2$.
In $L_2(7)$, there are two conjugate classes that are isomorphic to $S_4$.
Since all Sylow 2-subgroups are conjugate,
we get $R$ is contained in two non-conjugate subgroups~$A\le G$ and~$B\le G$
that are isomorphic to~$S_4$. Since~$A=RC_3$, $B=RC_3^x$, $x\in G$,
we obtain $G=\langle A,B\rangle \le P_G(R)$. So $R$ is permutable in~$G$.
If~$R<U<G$, then $U\cong S_4$, therefore $R$ is strongly permutable in~$G$.

Let $G\cong L_2(11)$. Then $|G|=2^2\cdot 3\cdot 5\cdot 11$,
$p=11$, $H\cong L_2(5)$. In $L_2(5)$,
only a Sylow 2-subgroup is $\mathbb{P}$-subnormal,
therefore~$r=2$. But $R$ is not permutable in $H\cong L_2(5)$,
hence $R$ is not strongly permutable in~$G\cong L_2(11)$.

Let $G\cong L_3(3)$. Then $|G|=2^4\cdot 3^3\cdot 13$,
$p=13$ and $H\cong M_9:S_3\cong C_3^2:GL_2(3)$.
Since $|H|=2^4\cdot 3^3$ and~$H$ is not $3$-closed,
we have~$r\ne 3$ by Lemma~\ref{KPMax} and~$r=2$.
But a Sylow $2$-subgroup $R$ is not $\mathbb{P}$-subnormal in~$H$ according to~\cite{atl}. Therefore in~$G\cong L_3(3)$, there are no $\mathbb{P}$-subnormal Sylow subgroups.

Let $G\cong L_3(5)$. Then $|G|=2^5\cdot 3\cdot 5^3 \cdot 31$,
$p=31$ and $H\cong C_5^2:GL_2(5)$. Since $|H|=2^5\cdot 3\cdot 5^3$
and~$H$ is not $5$-closed, we get~$r\ne 5$ by Lemma~\ref{KPMax} and~$r\in \{2,3\}$.
But a Sylow $2$--subgroup and $3$-subgroup
are not $\mathbb{P}$-subnormal in~$H$ according to~\cite{atl}.
Therefore in~$G\cong L_3(5)$, there are no $\mathbb{P}$-subnormal Sylow subgroups.

Let $G\cong L_2(2^m)$, where $2^m+1$ is prime.
Then $|G|=2^m(2^m-1)(2^m+1)$, $p=2^m+1$ and $H=N_G(Q)\cong C_2^m:C_{2^m-1}$,
$Q\cong C_2^m$ is a Sylow $2$-subgroup of ~$G$.
Since $Q$ is $\mathbb{P}$-subnormal in~$H$,
we deduce $Q$ is $\mathbb{P}$-subnormal in~$G$.
Suppose that $\langle g\rangle Q=Q\langle g\rangle$ for some~$g\in G$.
Then $\langle g\rangle Q\le N_G(Q)$ according to~\cite[II.8.27]{hup}.
Hence~$P_G(Q)=N_G(Q)$ and a Sylow $2$-subgroup of $G\cong L_2(2^m)$
is $\mathbb{P}$-subnormal in~$G$, but it is not permutable in~$G$.
Suppose that~$r\ne 2$. Then $R\le N_G(Q)=H$ and~$R$ is $\mathbb{P}$-subnormal
in~$RQ$ by Lemma~\ref{KPProp}\,(2), since~$R$ is $\mathbb{P}$-subnormal in~$H$
and~$H$ is soluble. By Lemma~\ref{KPMax}, $R$ is normal in~$RQ$ and~$R\le C_G(Q)$,
which is impossible  in~$G\cong L_2(2^m)$.
\end{proof}

\begin{corollary} [{\cite[Theorem 2.1]{km2020umj}}]\label{c01}
 If a Sylow $r$-subgroup of a group $G$ is $\mathbb{P}$-subnormal and~$r>2$,
 then~$G$ is $r$-soluble.
\end{corollary}

\begin{proof}
By Theorem~\ref{th0}\,(1), $G$ is not simple.
Let~$N$ be a normal subgroup of~$G$, $1\ne N\ne G$.
Then~$R\cap N$ is a Sylow $r$-subgroup of~$N$
and~$R\cap N$ is $\mathbb{P}$-subnormal in~$N$
in view of Lemma~\ref{KPProp}\,(1).
By induction, $N$ is $r$-soluble.
Note that $RN/N$ is a Sylow $r$-subgroup of~$G/N$
and~$RN/N$ is $\mathbb{P}$\nobreakdash-\hspace{0pt}subnormal in~$G/N$
in view of Lemma~\ref{KPProp}\,(1). By induction, $G/N$ is $r$-soluble.
Therefore~$G$ is $r$-soluble.
\end{proof}

\begin{corollary} [{\cite[Corollary 2.1.1]{km2020umj}}]\label{c02}
 If a Sylow 3-subgroup and Sylow 5\nobreakdash-\hspace{0pt}subgroup of a group $G$
 is $\mathbb{P}$-subnormal, then~$G$ is soluble.
\end{corollary}

\begin{proof}
By Corollary~\ref{c01}, $G$ is 3-soluble and 5-soluble.
Hence $G$ has a normal series, in which factors are 3-groups,
5\nobreakdash-\hspace{0pt}groups or $\{3,5\}^\prime$\nobreakdash-\hspace{0pt}groups.
Since $\{3,5\}^\prime$-groups are soluble~\cite[Theorem, p.18]{gor},
we conclude $G$ is soluble.
\end{proof}

\section{Groups with permutable and $\mathbb{P}$-subnormal primary cyclic subgroups}

Let $\mathfrak{F}$ be a class of groups.
A group $G$ is called a minimal non-$\mathfrak{F}$-group if
$G\notin\mathfrak{F}$ but every proper subgroup of $G$ belongs to $\mathfrak{F}$.
Minimal non-$\mathfrak{N}$-groups are also called Schmidt groups.
We remind the properties of Schmidt groups and minimal non-supersoluble groups
we need.

\begin{lemma}[{\cite[Theorem~1.1, 1.2, 1.5]{umk},\cite[Theorem~3]{Ball_Sch}}]\label{ls1}
Let $S$ be a Schmidt group. Then the following statements hold.

$(1)$~$S=P\rtimes Q$, where $P$ is a normal Sylow $p$-subgroup
and $Q$ is a non-normal Sylow $q$-subgroup, $p$ and $q$ are different primes and

\ \ $(1.1)$~if $P$ is abelian, then $P$ is elementary abelian of order~$p^m$,
            where $m$ is the order of $p$ modulo $q$;

\ \ $(1.2)$~if $P$ is not abelian, then $Z(P)=P'=\Phi (P)$ and $|P/Z(P)| = p^m$;

\ \ $(1.3)$~if $p>2$, then $P$ has the exponent~$p$;
            for $p=2$, the exponent of~$P$ is not more than~$4$.

\ \ $(1.4)$~$Q=\langle y\rangle$ is a cyclic subgroup and $y^q\in Z(S)$.

$(2)$~$G$ has exactly two classes of conjugate maximal subgroups
\[
\{P\times \langle y^q\rangle\},~~ \{\Phi (P)\times \langle x^{-1}yx\rangle~|~x\in P\setminus \Phi (P)\}.
\]
\end{lemma}

\begin{lemma}[{\cite{D66,BallR07}}]\label{minu}
Let $G$ be a minimal non-supersoluble group.
Then the following statements hold.

$(1)$~$G$ is soluble and $|\pi(G)| \leq 3$;

$(2)$~If $G$ is not a Schmidt group,
 then $G$ has a Sylow tower of supersolvable type;

$(3)$~$G$ has a unique normal Sylow subgroup $P$ and $P = G^\mathfrak{U}$;

$(4)$~$|P/\Phi(P)| > p$ and $P/\Phi(P)$ is a minimal normal subgroup of $G/\Phi(P)$;

$(5)$~If $\Phi(G)=1$, then $O_{p^\prime}(G)=1$
      and~$Q$ is either nonabelian of order $q^3$ and exponent $q$,
      or $Q$ is a cyclic $q$-group,
      or $Q$ is a $q$-group with a cyclic subgroup of index~$q$,
      or $Q$ is a supersoluble Schmidt group.
\end{lemma}

\begin{lemma}[{\cite[Lemma~1]{M95}}]\label{ls2}
 Let $S=P\rtimes Q$ be a supersoluble Schmidt group.
 Then $P=\langle x\rangle$ is a normal subgroup of order~$p$,
 $Q=\langle y\rangle$ is a cyclic subgroup of order~$q^b$,
 where $q$ divides $p-1$.
\end{lemma}

\begin{lemma}\label{lq}
 Let $G=P\rtimes Q$ be a Schmidt group, $Q=\langle y\rangle$.
 If $x\in G$ and $|y|$ does not divide $|x|$,
 then $x\in P\times \langle y^q\rangle$.
\end{lemma}

\begin{proof}
 Let $\langle x\rangle=\langle x_1\rangle\times\langle x_2\rangle$,
 where $\langle x_1\rangle$ is a Sylow $p$-subgroup and
 $\langle x_2\rangle$ is a Sylow $q$-subgroup of $\langle x\rangle$.
 Since $P$ is normal in $G$, $\langle x_1\rangle\le P$.
 Let $\langle x_2\rangle\leq Q^g=\langle y^g\rangle$, $g\in G$.
 As $|y|$ does not divide $|x_2|$, we conclude $\langle x_2\rangle<\langle y^g\rangle$
 and $x_2\in \langle (y^g)^q\rangle$. But $y^q\in Z(G)$, therefore
 \[
 (y^g)^q=\underbrace{y^g\cdot y^g\cdot \ldots \cdot y^g}_{q}=g^{-1}y^q g=y^q
 \]
 and $x_2\in\langle y^q\rangle$.
 Consequently, $\langle x\rangle\leq P\times \langle y^q\rangle$.
\end{proof}

\begin{lemma}\label{ls}
$(1)$ In a supersoluble Schmidt group, every subgroup is strongly permutable.

$(2)$ Let $G=P\rtimes Q$ be a non-supersoluble Schmidt group.
Then the following statements hold.

\ \ $(2.1)$ $Q$ is not permutable and $N_G(Q)=P_G(Q)=\Phi (P)\times Q\lessdot \, G$.

\ \ $(2.2)$ If $H\leq P$ and $P_G(H)=G$, then either $H=P$ or $H\leq \Phi(G)$.

\ \ $(2.3)$ Every primary permutable subgroup is normal in~$G$,
            and so it is strongly permutable in~$G$.

$(3)$ In Schmidt group~$G$, every subgroup of prime order and
every cyclic subgroup of order~4 is strongly permutable if and only if $G$ is supersoluble.
\end{lemma}

\begin{proof}
In view of Lemma~\ref{ls1}\,(2), maximal subgroups of $G$ are reduced to
$N_G(Q^g)=\Phi(P)\times Q^g$, $g\in G$ and  $P\times \langle y^q\rangle\lhd G$,
$\langle y\rangle =Q$.

$(1)$ Let $G=P\rtimes Q$ be a supersoluble Schmidt group.
Then $|P|=p$ and $q$ divides $p-1$, where $|Q|=q^b$ by Lemma~\ref{ls2}.
It is clear that $P$ and $Q$ are strongly permutable in~$G$. Let $Q_1\lessdot \, Q$.
Note that $P\times Q_1$ is cyclic and normal in $G$. Hence all subgroups of $P\times Q_1$
is normal in $G$ and strongly permutable.

$(2)$ Let $G=P\rtimes Q$ be a non-supersoluble Schmidt group.

$(2.1)$ Suppose that there is $x\in G\setminus N_G(Q)$ such that
$\langle x \rangle Q=Q\langle x \rangle$.
Let $\langle x\rangle=\langle a\rangle\langle b\rangle$,
where $\langle a\rangle$ is a Sylow $p$-subgroup and
$\langle b\rangle$ is a Sylow $q$-subgroup of $\langle x\rangle$.
Since $\langle b\rangle Q=Q$ according to~\cite[VI.4.7]{hup},
\[
\langle x \rangle Q=\langle a\rangle Q=\langle a\rangle \rtimes Q,
\ a\in N_G(Q)=\Phi(P)\times Q,
\]
So, $x\in N_G(Q)$ and $N_G(Q)=P_G(Q)\lessdot \,G$.

$(2.2)$ Assume that $H\leq P$ and $P_G(H)=G$. Since $P$ is normal in $G$, then $P_G(P)=G$.
Let
\[
H<P, \ x_i\in P_G(H), \ \langle x_i\rangle H=H\langle x_i\rangle,
\ i=1,\ldotp,n,
\ P_G(H)=\langle x_1,x_2,\ldots , x_n\rangle .
\]
If $|Q|$ does not divide $|x_i|$ for every~$i$,
then $P_G(H)\leq P\times \langle y^q\rangle<G$ by Lemma~\ref{lq},
a contradiction with the hypothesis. Therefore there is $x_j$ such that
$\langle x_j\rangle H=H \langle x_j\rangle$ and $|Q|$ divides $|x_j|$.
Hence $\langle x_j\rangle=\langle u\rangle \times Q^g$ for some $u\in P$ and $g\in G$.
Since $\langle x_j\rangle <G$,
\[
\langle x_j\rangle=\langle u\rangle \times Q^g\le
\Phi (P) \times Q^g.
\]
Hence $u\in\Phi(P)$. If $\langle x_j\rangle H=G$,  then $\langle u\rangle H=P$ and $H=P$,
a contradiction. So, $\langle x_j\rangle H<G$ and
$\langle x_j\rangle H\leq \Phi(P)\times Q^g$, hence $H\leq\Phi(P)$.
Since $P$ is normal in $G$, we conclude $H\leq\Phi(P)\leq \Phi(G)$.

$(3)$ Suppose that in a Schmidt group~$G=P\rtimes Q$, every cyclic subgroup
of prime order and cyclic subgroup of order~4 is strongly permutable. Statement~$(2.2)$
implies  $|P|=p$ and~$G$ is supersoluble.
Conversely, if~$G$ is a supersoluble Schmidt group, then by Statement~$(1)$,
every subgroup of prime order and every cyclic subgroup of order~4 is strongly permutable.
\end{proof}

\begin{lemma}\label{expp}
 Let $H$ be a  $p$-group of exponent~$p$ and $x\notin Z(H)$.
 Then $N_H(\langle x\rangle)=P_H(\langle x\rangle)$ and
 $\langle x\rangle$ is not permutable in $H$.
\end{lemma}

\begin{proof}
It is clear that $N_H(\langle x\rangle)\leq P_H(\langle x\rangle)$. Choose
$y\in H\setminus \langle x\rangle$ such that $\langle x\rangle\langle y\rangle=\langle y\rangle\langle x\rangle$.
Since $H$ is a $p$-group of exponent~$p$, we get $|\langle x\rangle\langle y\rangle|=p^2$.
Consequently, $H$ is abelian and
$\langle x\rangle\langle y\rangle=\langle x\rangle\times\langle y\rangle$,
and so $y\in N_H(\langle x\rangle)$ and $N_H(\langle x\rangle)=P_H(\langle x\rangle)$.
As $x\notin Z(H)$, we have $H\neq N_H(\langle x\rangle)=P_H(\langle x\rangle)$
and $\langle x\rangle$ is not permutable in $H$.
\end{proof}

\begin{theorem}\label{tc}
 If all primary cyclic subgroups of a group $G$ are
 strongly permutable, then~$G$ is supersoluble.
\end{theorem}

\begin{proof}
We use induction on the group order.
In view of Lemma~\ref{PerProp}\,(3) and by induction,
all proper subgroups of $G$ are supersoluble.
Hence $G=P\rtimes S$ is a minimal non-supersoluble group,
$P=G^\mathfrak{U}$. By Lemma~\ref{ls}\,(3), $G$ is not a Schmidt group,
therefore $G$ has a Sylow tower of supersoluble type and
$P$ is a Sylow $p$-subgroup of $G$ for $p=\max \ \pi(G)$.
In particular, $p>2$ and all nontrivial elements in $P$ are of order~$p$
by Lemma~\ref{minu}.

From Lemma~\ref{expp} it follows that $P$ is an elementary abelian $p$-subgroup,
and by Lemma~\ref{minu}, $P$ is a minimal normal subgroup in~$G$.

Assume that $N$ is a normal subgroup of~$G$, $N\neq 1$,
and $V/N$ is a cyclic $t$-subgroup, $t\in\pi(G)$.
Let $U$ be a subgroup of least order such that $U\leq V$, $UN=V$.
Then $U\cap N\leq\Phi(U)$, $V/N=UN/N\cong U/U\cap N$,
therefore $U$ is a cyclic $t$-subgroup. By the hypothesis,
$U$ is strongly permutable in $G$, and by Lemma~\ref{PerProp}\,(1),
$V/N$ is strongly permutable in $G/N$. By induction, $G/N$ is supersoluble,
hence $\Phi(G)=1$. From Lemma~\ref{minu}\,(5) it follows that $S$ is
either a cyclic primary group or supersoluble Schmidt group,
and~$O_{p^\prime}(G)=1$.

Let $S$ be either a cyclic primary group or $|\pi(S)|=2$.
In that case, all Sylow subgroups of $S$ are cyclic.
Assume that $A\leq P$, $|A|=p$ and $g\in G\setminus N_G(A)$ such that
$\langle g\rangle A\leq G$. Since $G=PS$, we conclude $g=bx$, $b\in P$, $x\in S$
and $\langle g\rangle=\langle b\rangle\times\langle x\rangle$.
If $x=1$, then $g=b\in P\leq N_G(A)$, a contradiction.
If $b=1$, then $g=x$ and $\langle g\rangle A=A\rtimes \langle g\rangle$,
since $\langle g\rangle A$ is $p$-closed. So, $g\in N_G(A)$, a contradiction.
Thus, $b\neq 1$, $x\neq 1$, $S^b\neq S$, $x^b=x\in S\cap S^b=D\neq 1$.

If $S$ is abelian, then $D\lhd \langle S,S^b\rangle=G$, $D\leq O_{p'}(G)=1$,
a contradiction. Therefore $S$ is not abelian and $S=Q\rtimes R$
is supersoluble Schmidt group by Lemma~\ref{minu} in view of~$\Phi(G)=1$.
Now, $|Q|=q$, $R=\langle y\rangle$ is an $r$-subgroup, $r$ divides $q-1$ and $y^r\in Z(S)$.

If $q$ divides $|D|$, then $Q\leq D$ and $Q\lhd\langle S,S^b\rangle =G$, a contradiction.
So, $D$ is an $r$-subgroup. If $y^r\neq 1$, then $D_1=D\cap \langle y^r\rangle\neq 1$ and
$D_1\lhd \langle S,S^b\rangle =G$, a contradiction. Consequently,
$y^r=1$, $D=\langle y\rangle=R$ and $|R|=r$.
Since $D=\langle x\rangle$, we get $\langle x\rangle=R\leq N_G(\langle b\rangle)$.

If $\langle b\rangle\lhd PQ$, then $N_G(\langle b\rangle)\geq\langle PQ,R\rangle=G$
and $\langle b\rangle\lhd G$, a contradiction. Hence $\langle b\rangle$
is not normal in $PQ$ and there is $u\in PQ\setminus N_{PQ}(\langle b\rangle)$
such that $\langle b\rangle\langle u\rangle=\langle u\rangle\langle b\rangle\leq PQ$.
Let $u=cf$, $c\in P$, $f\in Q$. If $c=1$, then
$\langle b\rangle\langle u\rangle=\langle b\rangle\rtimes\langle f\rangle$
and $u=f\in N_{PQ}(\langle b\rangle)$, a contradiction. If $f=1$,
then $u=c\in P\leq N_{PQ}(\langle b\rangle)$, a contradiction.
Consequently,
$c\neq 1$, $f\neq 1$ and $\langle u \rangle=\langle c\rangle\times \langle f\rangle$, $\langle f\rangle=Q$, $Q=Q^c\leq S\cap S^c$ and $S^c\neq S$, since $c\in P$,
$c\notin N_G(S)=S$. But now $Q\lhd \langle S,S^c\rangle =G$,
a contradiction.

Finally, we consider the case when $S=R$ is noncyclic Sylow $r$-subgroup of $G$.
Assume that in $S$, there is a cyclic subgroup $Z$ of index~$r$.
Let $A\leq P$ such that  $|A|=p$ and let  $g\in G\setminus N_G(A)$ such that
$\langle g\rangle A\leq G$. Since $G=PR$, we deduce $g=bx$, $b\in P$, $x\in R$ and
$\langle g\rangle=\langle b\rangle\times\langle x\rangle$.
If $b=1$, then $g=\langle x\rangle$ and $\langle g\rangle A=A\rtimes \langle g\rangle$,
since $\langle g\rangle A$ is $p$-closed and $A$ is a Sylow $p$-subgroup
of $\langle g\rangle A$. Hence $g\in N_G(A)$, a contradiction.
If $x=1$, then $g=b\in P\leq N_G(A)$. This contradicts with the choice of~$g$.
So, $b\neq 1$, $x\neq 1$, $R^b\neq R$, $x^b=x\in R\cap R^b$.

If $\langle x_1\rangle=\langle x\rangle\cap Z\neq 1$,
then $\langle x_1\rangle\lhd R$. Since $x_1^b=x_1\in Z^b\lhd R^b$,
we get $\langle x_1\rangle\lhd R^b$.
Therefore $\langle x_1\rangle\lhd \langle R,R^b\rangle=G$.
This contradicts with $R_G=1$.
So, $\langle x\rangle\cap Z=1$ and $R=Z\rtimes\langle x\rangle$,
$|x|=r$, $x\in C_G(\langle b\rangle)\leq N_G(\langle b\rangle)$.

If $\langle b\rangle\lhd PZ$, then $\langle b\rangle\lhd G$,
a contradiction. Thus, $N_{PZ}(\langle b\rangle)<PZ$
and there is $u\in PZ\setminus N_{PZ}(\langle b\rangle)$
such that $\langle u\rangle\langle b\rangle\leq G$.
Since $u\in PZ$, we conclude $\langle u\rangle=\langle c\rangle\times\langle y\rangle$,
$\langle c\rangle\leq P$, $\langle y\rangle\leq Z$. Verification shows that $c\neq 1$,
$y\neq 1$. From $N_G(Z)=R$ it follows that $Z^c\neq Z$ and $y^c=y\in Z\cap Z^c$.
Now, $\langle y\rangle\lhd \langle R,R^c\rangle=G$, a contradiction.

If $S$ is not an abelian group of order~$r^3$ and exponent~$r$,
then by Lemma~\ref{expp}, $S$ contains a nonpermutable cyclic primary subgroup,
which contradicts with the choice of~$G$.
\end{proof}

\begin{theorem}\label{tcp}
 If every primary cyclic subgroup of a group $G$ is
 $\mathbb{P}$\nobreakdash-\hspace{0pt}subnormal or strongly permutable,
 then~$G\in \rm{v}\mathfrak U$.
\end{theorem}

\begin{proof}
We use induction on the group order.
Let $N$ be a normal subgroup of a group $G$, $N\neq 1$,
and let $\langle a\rangle$ be a cyclic primary subgroup of $N$.
By the choice of $G$, $\langle a\rangle$ is $\mathbb{P}$-subnormal
or strongly permutable in $G$. If $\langle a\rangle$ is
$\mathbb{P}$-subnormal in $G$, then by Lemma~\ref{KPProp}\,(1),
$\langle a\rangle$ is $\mathbb{P}$-subnormal in $N$.
If $\langle a\rangle$ is strongly permutable in $G$,
then by Lemma~\ref{PerProp}\,(3), $\langle a\rangle$ is strongly permutable in $N$.
Now assume that $A/N$ is a cyclic $t$-subgroup, $t\in\pi(G)$.
Let $B$ be a subgroup of least order such that $B\leq A$, $BN=A$.
Then $B\cap N\leq\Phi(B)$, $A/N=BN/N\cong B/B\cap N$,
hence $B$ is a cyclic $t$-subgroup. By the choice of $G$,
$B$ is $\mathbb{P}$-subnormal or strongly permutable in $G$.
As $A/N=BN/N$, according to Lemma~\ref{KPProp}\,(1) and Lemma~\ref{PerProp}\,(1),
$A/N$ is $\mathbb{P}$-subnormal or strongly permutable in $G$.
Thus the hypothesis holds for all normal subgroups of $G$
and all quotients subgroups.

Suppose that $G$ is a simple group. If every primary cyclic subgroup of $G$ is
strongly permutable, then $G$ is supersoluble by Theorem~\ref{tc}.
Consequently, $G$ contains a cyclic primary subgroup $A$ such that
$A$ is $\mathbb{P}$-subnormal in $G$. Since the unit subgroup is
$\mathbb{P}$-subnormal in~$A$, then it is $\mathbb{P}$-subnormal in~$G$.
According to~\cite{kaz}, \cite[p. 342]{cs},
$G$ is isomorphic to one of the following groups.
\[
L_2(7), \ L_2(11), \ L_3(3), \ L_3(5), \ L_2(2^m),
2^m+1 \mbox{ is prime}.
\]
In every of these groups, a Sylow $r$-subgroup $R$ is cyclic for $r=\max \ \pi(G)$.
By the choice of $G$, $R$ is $\mathbb{P}$-subnormal or strongly permutable in $G$.
If $R$ is $\mathbb{P}$-subnormal in $G$, then by Lemma~\ref{KPMax},
$R$ is normal in $G$. If $R$ is strongly permutable in $G$,
then in view of Lemma~\ref{PG=NG}, $R$ is normal in $G$.
Consequently, $R$ is normal in $G$ and $G$ is not a simple group, a contradiction.

Thus in $G$, there is a normal subgroup $N$, $N\neq 1$, and by induction,
$G/N\in\mathrm{v}\mathfrak{U}$ and $N\in\mathrm{v}\mathfrak{U}$.
Hence $G$ is soluble. In view of Lemma~\ref{KPProp}\,(2) and by induction,
every proper subgroup of $G$ belongs to~$\mathrm{v}\mathfrak{U}$ and
$G$ is a minimal non-$\mathrm{v}\mathfrak{U}$-group.
According to~\cite[Theorem~B\,(4)]{mkRic},
$G$ is a biprimary minimal non-supersoluble group
in which non-normal Sylow subgroups are cyclic.
Hence $G=R\rtimes Q$ is a group such that a Sylow $r$-subgroup $R$ is normal in $G$
and a Sylow $q$-subgroup $Q$ is  cyclic and $\mathbb{P}$-subnormal or strongly permutable
in~$G$ by the choice of~$G$. By  Corollary~\ref{BiU},
$G\in \mathfrak{U}\subseteq \mathrm{v}\mathfrak{U}$.
\end{proof}

\begin{example}
In $A_4$, every subgroup of order~2 is $\mathbb{P}$-subnormal,
but it is not permutable.
\end{example}

\begin{example}
In $L_2(7)$, every subgroup of order~3 is permutable,
but it is not $\mathbb{P}$-subnormal.
\end{example}



\end{document}